\documentclass[11pt]{amsart}
\usepackage{amsfonts,amssymb,amsmath,amsthm,mathrsfs}
\usepackage{url}
\usepackage{enumerate}
\usepackage{txfonts}

\newtheorem{thrm}{Theorem}[section]
\newtheorem{lem}[thrm]{Lemma}

\theoremstyle{definition}

\numberwithin{equation}{section}

\author{Zhang Lin, Zhu Jun and Wu Junde}
\address{
Zhang Lin\newline\indent Department of Mathematics\\Zhejiang
University\\Hangzhou\newline\indent People's Republic of China}
\email{godyalin@163.com;linyz@zju.edu.cn}
\address{
Zhu Jun\newline\indent Institute of Mathematics\\Hangzhou Dianzi
University\\Hangzhou\newline\indent People's Republic of China}
\address{
Wu Junde\newline\indent Department of Mathematics\\Zhejiang
University\\Hangzhou\newline\indent People's Republic of China}

\begin{document}

\title[All-derivable points in nest algebras]
{All-derivable points in nest algebras}

\thanks{This work is supported by the National Natural Science Foundation of China (No. 10771191).}

\begin{abstract}

Suppose that $\mathscr{A}$ is an operator algebra on a Hilbert space
$H$. An element $V$ in $\mathscr{A}$ is called an all-derivable
point of $\mathscr{A}$ for the strong operator topology if every
strong operator topology continuous derivable mapping $\varphi$ at
$V$ is a derivation. Let $\mathscr{N}$ be a complete nest on a
complex and separable Hilbert space $H$. Suppose that $M$ belongs to
$\mathscr{N}$ with $\{0\}\neq M\neq\ H$ and write $\widehat{M}$ for
$M$ or $M^{\bot}$. Our main result is: for any $\Omega\in
alg\mathscr{N}$ with $\Omega=P(\widehat{M})\Omega P(\widehat{M})$,
if $\Omega |_{\widehat{M}}$ is invertible in
$alg\mathscr{N}_{\widehat{M}}$, then $\Omega$ is an all-derivable
point in $alg\mathscr{N}$ for the strong operator topology.

\end{abstract}
\maketitle
\section{Introduction}
Let $K$ and $H$ be complex and separable Hilbert spaces of
dimensions greater than one. $B(K,H)$ and $F(K,H)$ stand for the set
of all bounded linear operators and the set of all finite rank
operators from $K$ into $H$, respectively. When $H=K$, $B(K,H)$ and
$F(K,H)$ are abbreviated to $B(H)$ and $F(H)$, respectively. The
adjoint operator of $T$ is denoted by $T^{*}$. Suppose $x\in K$ and
$y\in H$. The rank one operator $<\cdot,x>y$, from $K$ into $H$, is
denoted by $y\otimes x$. If $\mathscr{N}$ is a complete \emph{nest}
on $H$, then the \emph{nest algebra} $alg\mathscr{N}$ is the Banach
algebra of all bounded linear operators which leave every member of
$\mathscr{N}$ invariant. For $N\in \mathscr{N}$, $N_{-}$ stands for
$\vee\{M\in\mathscr{N}:M\subset N\}$, and $\mathscr{N}_N$ stands for
the nest $\{L\cap N: L\in\mathscr{N}\}$ in $N$. We write $P(N)$ for
the orthogonal projection operator from $H$ onto $N$. The identity
of $B(N)$ is denoted by $I_{N}$ and the restriction of an operator
$T\in B(H)$ to the subspace $N$ is denoted by $T|_{N}$.\\

Suppose that $\mathscr{A}$ is a subalgebra of $B(H)$ and $V$ is an
operator in $\mathscr{A}$. A linear mapping $\varphi$ from
$\mathscr{A}$ into itself is called a \emph{derivable mapping} at
$V$ if $\varphi(ST)=\varphi(S)T+S\varphi(T)$ for any $S,T$ in
$\mathscr{A}$ with $ST=V$. Operator $V$ is called an
\emph{all-derivable point} in $\mathscr{A}$ for the strong operator
topology if every strong operator topology continuous derivable
mapping $\varphi$ at $V$ is a derivation.\\

In recent years the study of all-derivable points in operator
algebras has attracted many researchers' attentions. Jing, Lu, and
Li \cite{Jing} proved that every derivable mapping $\varphi$ at $0$
with $\varphi(I)=0$ on nest algebras is a derivation. Li, Pan, and
Xu \cite{Pan} showed that every derivable mapping $\varphi$ at $0$
with $\varphi(I)=0$ on CSL algebras is a derivation. Zhu and Xiong
proved the following results in \cite{Zhu1,Zhu2,Zhu3,Zhu6,Zhu7}: 1)
every norm-continuous generalized derivable mapping at $0$ on some
CSL algebras is a generalized derivation; 2) every invertible
operator in nest algebras is an all-derivable point for the strong
operator topology; 3) $V$ is an all-derivable point in
$\mathscr{TM}_n$ if and only if $V\neq 0$, where $\mathscr{TM}_n$ is
the algebras of all $n\times n$ upper triangular matrices; and 4)
every orthogonal projection operator $P(M)(\{0\}\neq
M\in\mathscr{N})$ is an all-derivable point in nest algebra
$alg\mathscr{N}$ for the strong operator topology.\\

The main purpose of this paper is to study the all-derivable points
in nest algebras. Suppose that $M$ belongs to $\mathscr{N}$ with
$\{0\}\neq M\neq\ H$ and write $\widehat{M}$ for $M$ or $M^{\bot}$.
We shall prove: for any $\Omega\in alg\mathscr{N}$ with
$\Omega=P(\widehat{M})\Omega P(\widehat{M})$, if $\Omega
|_{\widehat{M}}$ is invertible in $alg\mathscr{N}_{\widehat{M}}$,
then $\Omega$ is an all-derivable point in $alg\mathscr{N}$ for the
strong operator topology.

\section{Three lemmas}
It is known that every operator $S$ in $B(H)$ can be uniquely
expressed in the form of a $2\times 2$ operator matrix relative to
the orthogonal decomposition $H=M\oplus M^{\bot}$. Thus we
immediately get the following
proposition.\\
\textbf{Proposition 2.0} \emph{Let $\mathscr{N}$ be a complete nest
on a complex and separable Hilbert space $H$. For an arbitrary $M$
in $\mathscr{N}$ with $\{0\}\neq M\neq H$, we have}
$$alg\mathscr{N}=\left\{\left[\begin{array}{ccc}
              X & Y\\
              0 & Z
       \end{array}
 \right]: X\in alg\mathscr{N}_{M}, Z\in alg\mathscr{N}_{M^{\bot}}, Y\in B(M^{\bot},M)\right\}.$$
The following three lemmas will be used to prove the main result of
this paper in Section \ref{sect3}.\\

\begin{lem}\label{l1} Let $H$ be a complex and separable Hilbert space and let
$\mathscr{N}$ be a complete nest in $H$. Suppose that $\delta$ is a
strong operator topology continuous linear mapping from
$alg\mathscr{N}$ into itself and $\Gamma$ is an invertible operator
in $alg\mathscr{N}$. If the following equation
\begin{eqnarray}
\delta(\Gamma)=\delta(\Gamma S_{1})S_{2}+\Gamma
S_{1}\delta(S_{2})\label{eqn2.1}
\end{eqnarray}
holds for any $S_{1},S_{2}$ in $alg\mathscr{N}$ with $S_{1}S_{2}=I$,
then $\delta$ is an inner derivation.
\end{lem}
\begin{proof} Put $S_{1}=S_{2}=I$ in Eq.~(\ref{eqn2.1}), we have
$S_{1}S_{2}=I$. It follows that $\Gamma\delta(I)=0$. That is,
$\delta(I)=0$ since $\Gamma$ is invertible in $alg\mathscr{N}$. Put
$S_{1}=I-aP$ and $S_{2}=I-bP$ in Eq.~(\ref{eqn2.1}), where $P$ is an
idempotent in $alg\mathscr{N}$ and $a,b$ are two complex numbers
such that $a+b=ab=1$. Thus we get that $S_{1}S_{2}=I$. Furthermore,
we have
\begin{eqnarray*}
\delta(\Gamma)&=&\delta(\Gamma -a\Gamma P)(I-bP)+(\Gamma -a\Gamma
P)\delta(I-bP)\\
&=&[\delta(\Gamma)-b\delta(\Gamma)P-a\delta(\Gamma
P)+ab\delta(\Gamma P)P]+[\Gamma\delta(I)-b\Gamma\delta(P)-a\Gamma
P\delta(I)+ab\Gamma P\delta(P)].
\end{eqnarray*}
It follows from $\delta(I)=0$ that
\begin{eqnarray}
a\delta(\Gamma P)+b[\delta(\Gamma)P+\Gamma\delta(P)]=\delta(\Gamma
P)P+\Gamma P\delta(P).\label{eqn2.2}
\end{eqnarray}
Interchanging the position of $a$ and $b$ in Eq.~(\ref{eqn2.2}), we
have
\begin{eqnarray}
b\delta(\Gamma P)+a[\delta(\Gamma)P+\Gamma\delta(P)]=\delta(\Gamma
P)P+\Gamma P\delta(P).\label{eqn2.3}
\end{eqnarray}
It follows from Eq.~(\ref{eqn2.2}) and Eq.~(\ref{eqn2.3}) that
$\delta(\Gamma P)=\delta(\Gamma)P+\Gamma\delta(P)$. Notice that
every rank-one operator in $alg\mathscr{N}$ may be written as a
linear combination of at most four idempotents in $alg\mathscr{N}$
(see \cite{Hadwin}) and every finite rank operator in
$alg\mathscr{N}$ may be represented as a sum of rank-one operators
in $alg\mathscr{N}$ (see \cite{Erdos}). Thus we obtain
\begin{eqnarray}
\delta(\Gamma F)=\delta(\Gamma)F+\Gamma\delta(F)\label{eqn2.4}
\end{eqnarray}
for any $F$ in $F(H)\cap alg\mathscr{N} $. By applying Erd\"{o}s
Density Theorem (see \cite{Erdos}) to Eq.~(\ref{eqn2.4}), we obtain
that $\delta(\Gamma R)=\delta(\Gamma)R+\Gamma\delta(R)$ for any $R$
in $alg\mathscr{N}$. In particular,
\begin{eqnarray}
\delta(\Gamma
S_{1})=\delta(\Gamma)S_{1}+\Gamma\delta(S_{1}).\label{eqn2.5}
\end{eqnarray}
It follows from Eq.~(\ref{eqn2.1}) and Eq.~(\ref{eqn2.5}) that
$\Gamma[\delta(S_{1})S_{2}+S_{1}\delta(S_{2})]=0$. Thus we get that
$\delta(S_{1})S_{2}+S_{1}\delta(S_{2})=0$ for any $S_{1},S_{2}$ in
$alg\mathscr{N}$ with $S_{1}S_{2}=I$ since $\Gamma$ is invertible in
$alg\mathscr{N}$. Note that $\delta(I)=0$, and so $\delta$ is a
derivable mapping at $I$. It follows that $\delta$ is an inner
derivation (see \cite{Zhu2}).
\end{proof}
\begin{lem}\label{l2} Let $K,H$ be two complex and separable Hilbert spaces with
dimensions greater than one, and let $\mathscr{N}$ and
$\mathscr{N'}$ be two complete nests in $H$ and $K$, respectively.
Suppose that $\varphi$ is a strong operator
topology continuous linear mapping from $B(K,H)$ into itself. \\~\\
(1) If $\ T\varphi(S)=0$ for any $T$ in $alg\mathscr{N}$ and $S$ in
$B(K,H)$ with $TS=0$, then there exists an operator $D$ in $B(K)$
such that $$\varphi(S)=SD$$ for any
$S$ in $B(K,H)$.\\~\\
(2) If $\ \varphi(S)T'=0$ for any $T'$ in $alg\mathscr{N'}$ and $S$
in $B(K,H)$ with $ST'=0$, then there exists an operator $D'$ in
$B(H)$ such that $$\varphi(S)=D'S$$ for any $S$ in $B(K,H)$.
\end{lem}
\begin{proof} We only prove (1). One can prove (2) similarly.\\

\emph{Case 1.} Suppose that $\{0\}_{+}\neq \{0\}$. For any $x$ in
$H$ and $g$ in $K$, it is clear that $x\otimes g$ is in $B(K,H)$.
For an arbitrary vector $z$ in $\{x\}^{\bot}(\subset
H=\{0\}^{\bot})$ and $y$ in $\{0\}_{+}$ with $y\neq 0$, $y\otimes z$
in $alg\mathscr{N}$ and $(y\otimes z)(x\otimes g)=0$. Under the
hypothesis, we get that $y\otimes z\varphi(x\otimes g)=0$. Thus
there exists a vector $\omega_{x,g}$ in $K$ such that
$\varphi(x\otimes g)=x\otimes \omega_{x,g}.$ It follows that
$$\varphi(x\otimes (f+g))=x\otimes \omega_{x,f+g}.$$ for any $f,g$ in $K$. Moreover, $$\varphi(x\otimes (f+g))=\varphi(x\otimes f)+\varphi(x\otimes
g)=x\otimes \omega_{x,f}+x\otimes \omega_{x,g}.$$ Thus
$\omega_{x,f+g}=\omega_{x,f}+\omega_{x,g}.$ Similarly, we obtain
that $\omega_{x,\lambda g}=\lambda\omega_{x,g}$. Consequently, there
exists a linear mapping $L_{x}$ from $K$ into $K$ such that
$\varphi(x\otimes g)=x\otimes L_{x}g$. Furthermore, we have
$$(x+v)\otimes L_{x+v}g=\varphi((x+v)\otimes g)=\varphi(x\otimes g)+\varphi(v\otimes g)=x\otimes L_{x}g+v\otimes L_{v}g$$
for any $v$ in $H$. So $x\otimes (L_{x+v}-L_{x})g+v\otimes
(L_{x+v}-L_{v})g=0$, which implies that $L_{x}=L_{v}$ when $x$ and
$v$ are linearly independent. If $x$ and $v$ are linearly dependent,
there exists some complex number $t$  such that $v=tx$. Since
$dimH>1$, a vector $u$ can be chosen from $H$ such that $u$ and $x$
are linearly independent. Thus $L_{x}=L_{u}=L_{v}$ since $u$ and $v$
are linearly independent. This implies that $L_{x}$ is independent
of $x$ for any $x\in H$. If we write $L=L_{x}$, then
$\varphi(x\otimes g)=x\otimes Lg$ for any $x$ in $H$ and $g$ in $K$.
Next we shall prove that $L$ is in $B(K)$. In fact, for arbitrary
sequence $(g_{n})$ in $K$ with $g_{n}\to g$ and $Lg_{n}\to h$, we
have $x\otimes Lg_{n}=\varphi(x\otimes g_{n})\to \varphi(x\otimes
g)=x\otimes Lg$. So $x\otimes (Lg-h)=0$, namely $Lg=h$. Therefore
$L$ is a closed operator. By the Closed Graph Theorem, we obtain
that $L$ is a bounded linear operator on $K$. Since $L\in B(K)$,
$\varphi(x\otimes g)=x\otimes Lg=x\otimes gL^{*}$. We write $D$ as
$L^{*}$. So $\varphi(x\otimes g)=x\otimes gD$. Furthermore, we have
$$\varphi(FS)=FSD$$ for any $S$ in $B(K,H)$ and finite rank operator
$F$ in $alg\mathscr{N}$. Since $\varphi$ is a strong operator
topology continuous linear mapping, it follows from Erd\"{o}s
Density Theorem that $\varphi(S)=SD$ for any $S$ in $B(K,H)$.\\

\emph{Case 2.} Suppose that $\{0\}_{+}=\{0\}$. Then there exists a
sequence $\{N_{n}\}$ in $\mathscr{N}$ such that the following
statements hold: 1) $N_{1}^{\bot}\supsetneq N_{2}^{\bot}\supsetneq
\cdots \supsetneq N_{j}^{\bot}\supsetneq N_{j+1}^{\bot}\supsetneq
\cdots \supsetneq \{0\};$ 2) $\{P(N_{n}^{\bot})\}$ strongly
converges to $0$ as $n\rightarrow+\infty$. It is obvious that
$$N_{1}\subsetneq N_{2}\subsetneq\cdots\subsetneq N_{j}\subsetneq N_{j+1}\subsetneq\cdots\subsetneq
H$$ and the sequence $\{P(N_{n})\}$ strongly converges to the unit
operator $I_{H}$ as $n\rightarrow+\infty$. For an arbitrary integer
$n$ and $x$ in $N_{n}$, by imitating the proof of case 1, we can
find a linear mapping $D_{N_{n}}$ on $K$ such that $\varphi(x\otimes
g)=x\otimes gD_{N_{n}}$ for any $x$ in $N_{n}$ and $g$ in $K$. Note
that $N_{n}\subseteq N_{m}$ $(m>n)$ and $\varphi(x\otimes
g)=x\otimes gD_{N_{m}}$ for any $x$ in $N_{m}$ and $g$ in $K$. So
$x\otimes gD_{N_{n}}=x\otimes gD_{N_{m}}$ for any $x\in N_{n}$ and
$g\in K$. It follows that $D_{N_{n}}=D_{N_{m}}$. Hence $D_{N_{n}}$
is independent of $N_{n}$. We write $D$ as $D_{N_{n}}$. Thus
$\varphi(x\otimes g)=x\otimes gD$ for any $x$ in $N_{n}$ and g in
$K$. For any $x$ in $H$, put $x_{n}=P(N_{n})x$. Then we get that
$$\varphi(x_{n}\otimes g)=x_{n}\otimes gD.$$ That is,
$$\varphi(P(N_{n})x\otimes g)=P(N_{n})x\otimes gD.$$ Since $\varphi$
is a strong operator topology continuous linear mapping and
$P(N_{n})$ strongly converges to $I_{H}$ as $n\rightarrow+\infty$,
taking limit on both sides in the above equation, we obtain that
$\varphi(x\otimes g)=x\otimes gD$ for any $x$ in $H$ and $g$ in $K$.
The rest of the proof is similar to case 1. The lemma is proved.
\end{proof}
\begin{lem}\label{l3} Let $\mathscr{A}$ be an unital subalgebra of $B(H)$, where H is
a complex and separable Hilbert space. Suppose that $\phi$ is a
linear mapping from $\mathscr{A}$ into itself. If $\phi$ vanishes at
every invertible operator in $\mathscr{A}$, then $\phi$ vanishes on
$\mathscr{A}$.
\end{lem}
\begin{proof} We only need to prove that $\phi(T)=0$ for any operator
$T$ in $\mathscr{A}$. Take a complex number $\lambda$ with
$|\lambda|
>\|T\|$. It follows that $\lambda I-T$
is invertible in $\mathscr{A}$. We thus see that $\phi(\lambda
I-T)=0$ by the hypothesis. Thus we have $\phi(T)=\lambda\phi(I)$ by
the linearity of $\phi$. So $\phi(T)=0$ for any $T$ in
$\mathscr{A}$.
\end{proof}

\section{All-derivable points in $alg\mathscr{N}$}\label{sect3}

In this section, we always assume that $M$ belongs to $\mathscr{N}$
with $\{0\}\neq M\neq\ H$, and write $\widehat{M}$ for $M$ or
$M^{\bot}$. Throughout the rest of this paper, every upper
triangular $2\times 2$ operator matrix relative to the orthogonal
decomposition $H=M\oplus M^{\bot}$ always
stands for the element of nest algebra $alg\mathscr{N}$. The unit operator on $M$ is denoted by $I_{M}$. The following theorem is our main result.

\begin{thrm}\label{t1} Let $\mathscr{N}$ be a complete nest on a complex and
separable Hilbert space $H$. Suppose that $M$ belongs to
$\mathscr{N}$ with $\{0\}\neq M\neq\ H$ and write $\widehat{M}$ for
$M$ or $M^{\bot}$. For any $\Omega\in alg\mathscr{N}$ with
$\Omega=P(\widehat{M})\Omega P(\widehat{M})$, if $\Omega
|_{\widehat{M}}$ is invertible in $alg\mathscr{N}_{\widehat{M}}$,
then $\Omega$ is an all-derivable point in $alg\mathscr{N}$ for the
strong operator topology.
\end{thrm}
\begin{proof} Let $\varphi$ be a strong operator topology continuous
derivable linear mapping at $\Omega$ from $alg\mathscr{N}$ into
itself. We only need to show that $\varphi$ is a derivation. For
arbitrary $X$ in $alg\mathscr{N}_{M}$, $Y$ in $B(M^{\bot},M)$ and
$Z$ in $alg\mathscr{N}_{M^{\bot}}$, we write
$$\left\{\begin{array}{ccc}
          \varphi(\left[\begin{array}{ccc}X&0\\0&0\end{array}\right])&=&\left[\begin{array}{ccc}A_{11}(X)&A_{12}(X)\\0&A_{22}(X)\end{array}\right],\\
          \varphi(\left[\begin{array}{ccc}0&Y\\0&0\end{array}\right])&=&\left[\begin{array}{ccc}B_{11}(Y)&B_{12}(Y)\\0&B_{22}(Y)\end{array}\right],\\
          \varphi(\left[\begin{array}{ccc}0&0\\0&Z\end{array}\right])&=&\left[\begin{array}{ccc}C_{11}(Z)&C_{12}(Z)\\0&C_{22}(Z)\end{array}\right].
        \end{array}
\right.$$ Obviously, $A_{ij},B_{ij}$ and $C_{ij}(i,j=1,2,i\leq j)$
are strong operator topology continuous linear mappings on
$alg\mathscr{N}_{M},\ B(M^{\bot},M)$, and
$alg\mathscr{N}_{M^{\bot}}$, respectively. \\
\textsl{Case 1.} Suppose that $\widehat{M}=M$. Then $\Omega$ may be
represented as the following matrix relative to the orthogonal
decomposition $H=M\oplus M^{\bot}$:
$$\Omega=\left[\begin{array}{ccc}W&0\\0&0\end{array}\right],$$
where $W$ is an invertible operator in $alg\mathscr{N}_{M}$.
The proof are divided into the following five steps:\\

\emph{Step 1.} For arbitrary $X_{1},X_{2}$ in $alg\mathscr{N}_M$
with $X_{1}X_{2}=I_{M}$, taking $S=\left[\begin{array}{ccc}W
X_{1}&0\\0&0\\\end{array}\right]$ and $T=\left[\begin{array}{ccc}
X_{2}&0\\0&0\\\end{array}\right]$, then $ST=\Omega$. Since $\varphi$
is a derivable mapping at $\Omega$ on $alg\mathscr{N}$, we have
\begin{eqnarray*}
& &\left[\begin{array}{ccc}A_{11}(W)
&A_{12}(W)\\0&A_{22}(W)\\\end{array}\right]=\varphi(\Omega)=\varphi(S)T+S\varphi(T)\\
&=&\left[\begin{array}{ccc}A_{11}(W X_{1}) &A_{12}(W
X_{1})\\0&A_{22}(W
X_{1})\\\end{array}\right]\left[\begin{array}{ccc}
X_{2}&0\\0&0\\\end{array}\right]+\left[\begin{array}{ccc}W
X_{1}&0\\0&0\\\end{array}\right]\left[\begin{array}{ccc}A_{11}(X_{2})
&A_{12}(X_{2})\\0&A_{22}(X_{2})\\\end{array}\right]\\
&=&\left[\begin{array}{ccc}A_{11}(W X_{1})X_{2}+W
X_{1}A_{11}(X_{2})&W X_{1}A_{12}(X_{2})\\0&0\end{array}\right].
\end{eqnarray*}
Furthermore,
\begin{eqnarray}
A_{11}(W)&=&A_{11}(W X_{1})X_{2}+WX_{1}A_{11}(X_{2}),\label{eqn1}\\
A_{12}(W)&=&WX_{1}A_{12}(X_{2}),\label{eqn2}\\\nonumber
A_{22}(W)&=&0
\end{eqnarray}
for any $X_{1},X_{2}$ in $alg\mathscr{N}_M$ with $X_{1}X_{2}=I_{M}$.
By Lemma \ref{l1}, we get that $A_{11}$ is an inner derivation on
$alg\mathscr{N}_{M}$. Then there exists an operator $A\in
alg\mathscr{N}_{M}$ such that
\begin{eqnarray}
A_{11}(X)=XA-AX\label{eqn8}
\end{eqnarray}
for any $X$ in $alg\mathscr{N}_M$.\par For an arbitrary invertible
operator $X$ in $alg\mathscr{N}_M$, putting $X_{2}=X,X_{1}=X^{-1}$
in Eq.~(\ref{eqn2}), then we get that $A_{12}(W)=W X^{-1}A_{12}(X)$,
i.e., $W ^{-1} A_{12}(W)=X^{-1}A_{12}(X)$. Taking $X=I_{M}$, we have
$W^{-1}A_{12}(W)=A_{12}(I_{M})$. Thus we get that
$$A_{12}(X)=XA_{12}(I_M)$$ for any invertible operator $X$ in
$alg\mathscr{N}_M$. It follows from Lemma \ref{l3} that
$A_{12}(X)=XA_{12}(I_M)$ for any operator $X$ in $alg\mathscr{N}_M$.
If we write $B$ for $A_{12}(I_M)$, then
\begin{eqnarray}
A_{12}(X)=XB\label{eqn9}
\end{eqnarray}
for any $X$ in $alg\mathscr{N}_M$.\\

\emph{Step 2.} For arbitrary $Z_{1},Z_{2}$ in
$alg\mathscr{N}_{M^{\bot}}$ with $Z_{1}Z_{2}=0$ and $X_{1},X_{2}$ in
$alg\mathscr{N}_M$ with $X_{1}X_{2}=I_M$, taking
$S=\left[\begin{array}{ccc}W X_{1}&0\\0&Z_{1}\\\end{array}\right]$
and $T=\left[\begin{array}{ccc}
X_{2}&0\\0&Z_{2}\\\end{array}\right]$, then $ST=\Omega$. Thus we
have
\begin{eqnarray*}
& &\left[\begin{array}{ccc}A_{11}(W)
&A_{12}(W)\\0&A_{22}(W)\\\end{array}\right]=\varphi(\Omega)=\varphi(S)T+S\varphi(T)\\
&=&\left[\begin{array}{ccc}A_{11}(WX_{1})+C_{11}(Z_{1})&A_{12}(WX_{1})+C_{12}(Z_{1})\\0&A_{22}(W
X_{1})+C_{22}(Z_{1})\\\end{array}\right]\left[\begin{array}{ccc}
X_{2}&0\\~0&Z_{2}\\\end{array}\right]\\& &+\left[\begin{array}{ccc}W
X_{1}&0\\~~0&Z_{1}\\\end{array}\right]\left[\begin{array}{ccc}A_{11}(X_{2})+C_{11}(Z_{2})
&A_{12}(X_{2})+C_{12}(Z_{2})\\0&A_{22}(X_{2})+C_{22}(Z_{2})\\\end{array}\right]\\
&=&\left[\begin{array}{ccc}A_{11}(W
X_{1})X_{2}+C_{11}(Z_{1})X_{2}&A_{12}(W
X_{1})Z_{2}+C_{12}(Z_{1})Z_{2}\\+W X_{1}A_{11}(X_{2})+W
X_{1}C_{11}(Z_{2})&+W X_{1}A_{12}(X_{2})+W X_{1}C_{12}(Z_{2})\\
~&~\\
~&A_{22}(WX_{1})Z_{2}+C_{22}(Z_{1})Z_{2}\\0&+Z_{1}A_{22}(X_{2})+Z_{1}C_{22}(Z_{2})\end{array}\right].
\end{eqnarray*}
Furthermore,
\begin{eqnarray}
A_{11}(W)&=&A_{11}(W X_{1})X_{2}+C_{11}(Z_{1})X_{2}+W
X_{1}A_{11}(X_{2})+W
X_{1}C_{11}(Z_{2}),\label{eqn10}\\
A_{12}(W)&=&A_{12}(W X_{1})Z_{2}+C_{12}(Z_{1})Z_{2}+W
X_{1}A_{12}(X_{2})+W
X_{1}C_{12}(Z_{2}),\label{eqn11}\\
A_{22}(W)&=&A_{22}(W
X_{1})Z_{2}+C_{22}(Z_{1})Z_{2}+Z_{1}A_{22}(X_{2})+Z_{1}C_{22}(Z_{2})\label{eqn12}
\end{eqnarray}
for any $Z_{1},Z_{2}$ in $alg\mathscr{N}_{M^{\bot}}$ with
$Z_{1}Z_{2}=0$ and $X_{1},X_{2}$ in $alg\mathscr{N}_M$ with
$X_{1}X_{2}=I_M$. Substituting the expression of $A_{11}(W)$ in
Eq.~(\ref{eqn1}) into Eq.~(\ref{eqn10}), and the expression of
$A_{12}(W)$ in Eq.~(\ref{eqn9}) into Eq.~(\ref{eqn11}),
respectively, we have
\begin{eqnarray}
0&=&C_{11}(Z_{1})X_{2}+W X_{1}C_{11}(Z_{2}),\label{eqn13}\\
W B&=&W X_{1}BZ_{2}+C_{12}(Z_{1})Z_{2}+W B+W
X_{1}C_{12}(Z_{2})\label{eqn14}
\end{eqnarray}
for any $Z_{1},Z_{2}$ in $alg\mathscr{N}_{M^{\bot}}$ with
$Z_{1}Z_{2}=0$ and $X_{1},X_{2}$ in $alg\mathscr{N}_M$ with
$X_{1}X_{2}=I_M$. For an arbitrary $Z$ in
$alg\mathscr{N}_{M^{\bot}}$, Putting $Z_{1}=0$ and $Z_{2}=Z$, it
follows from Eq.~(\ref{eqn13}) and Eq.~(\ref{eqn14}) that
\begin{eqnarray*}
C_{11}(Z)&=&0\label{eqn15}
\end{eqnarray*}
and
\begin{eqnarray*}
C_{12}(Z)&=&-BZ\label{eqn16}
\end{eqnarray*}
for any $Z$ in $alg\mathscr{N}_{M^{\bot}}$. Taking
$Z_{1}=I_{M^{\bot}}$ and $Z_{2}=0$ in Eq.~(\ref{eqn12}), we get that
$A_{22}(X_{2})=A_{22}(W)$ for any $X_{2}$ in $alg\mathscr{N}_M$.
Furthermore, $A_{22}(X)=0$ for any invertible operator $X$ in
$alg\mathscr{N}_M$. It follows from Lemma \ref{l2} that
\begin{eqnarray}
A_{22}(X)=0\label{eqn17}
\end{eqnarray}
for any $X$ in $alg\mathscr{N}_M$. Using Eq.~(\ref{eqn17}) and
Eq.~(\ref{eqn12}), we get that
$C_{22}(Z_{1})Z_{2}+Z_{1}C_{22}(Z_{2})=0$, namely $C_{22}$ is a
derivable mapping at $0$.\\

\emph{Step 3.} For arbitrary $Y$ in $B(M^{\bot},M)$ and
$X_{1},X_{2}$ in $alg\mathscr{N}_M$ with $X_{1}X_{2}=I_M$, taking
$S=\left[\begin{array}{ccc}W X_{1}&Y\\0&0\\\end{array}\right]$ and
$T=\left[\begin{array}{ccc} X_{2}&0\\0&0\\\end{array}\right]$, then
$ST=\Omega$. Thus we have
\begin{eqnarray*}
& &\left[\begin{array}{ccc}A_{11}(W)
&A_{12}(W)\\0&0\\\end{array}\right]=\varphi(\Omega)=\varphi(S)T+S\varphi(T)\\
&=&\left[\begin{array}{ccc}A_{11}(W X_{1})+B_{11}(Y)&A_{12}(W
X_{1})+B_{12}(Y)\\0&B_{22}(Y)\\\end{array}\right]\left[\begin{array}{ccc}
X_{2}&0\\0&0\\\end{array}\right]\\& &+\left[\begin{array}{ccc}W
X_{1}&Y\\0&0\\\end{array}\right]\left[\begin{array}{ccc}A_{11}(X_{2})
&A_{12}(X_{2})\\0&0\\\end{array}\right]\\
&=&\left[\begin{array}{ccc}A_{11}(W X_{1})X_{2}+B_{11}(Y)X_{2}+W
X_{1}A_{11}(X_{2})&W X_{1}A_{12}(X_{2})\\0&0\end{array}\right].
\end{eqnarray*}
Furthermore,
\begin{eqnarray*}
A_{11}(W)=A_{11}(W X_{1})X_{2}+B_{11}(Y)X_{2}+W
X_{1}A_{11}(X_{2}).\label{eqn18}
\end{eqnarray*}
Since $A_{11}$ is a inner derivation and $X_{2}$ is an invertible
operator in $alg\mathscr{N}_M$, we have
\begin{eqnarray*}
B_{11}(Y)=0\label{eqn19}
\end{eqnarray*}
for any $Y$ in $B(M^{\bot},M)$.\\

\emph{Step 4.} For arbitrary $Y$ in $B(M^{\bot},M)$, $Z$ in
$alg\mathscr{N}_{M^{\bot}}$, taking $S=\left[\begin{array}{ccc}
I_{M}&Y\\0&0\\\end{array}\right]$ and $T=\left[\begin{array}{ccc} W
&-YZ\\0&Z\\\end{array}\right]$, then $ST=\Omega$. Thus we have
\begin{eqnarray*}
& &\left[\begin{array}{ccc}A_{11}(W)
&A_{12}(W)\\0&0\\\end{array}\right]=\varphi(\Omega)=\varphi(S)T+S\varphi(T)\\
&=&\left[\begin{array}{ccc}0&A_{12}(I_{M})+B_{12}(Y)\\0&B_{22}(Y)\\\end{array}\right]\left[\begin{array}{ccc}
W &-YZ\\0&Z\\\end{array}\right]\\&
&+\left[\begin{array}{ccc}I_{M}&Y\\0&0\\\end{array}\right]\left[\begin{array}{ccc}A_{11}(W)
&A_{12}(W)-B_{12}(YZ)+C_{12}(Z)\\0&C_{22}(Z)-B_{22}(YZ)\\\end{array}\right]\\
&=&\left[\begin{array}{ccc}A_{11}(W)&(A_{12}(I_{M})+B_{12}(Y))Z+A_{12}(W)-B_{12}(YZ)\\
~&+C_{12}(Z)+Y(C_{22}(Z)-B_{22}(YZ))\\
~&~\\0&B_{22}(Y)Z\end{array}\right].
\end{eqnarray*}
Furthermore,
\begin{eqnarray}
\nonumber A_{12}(W)&=&(A_{12}(I_{M})+B_{12}(Y))Z+A_{12}(W)\\
&&-B_{12}(YZ)+C_{12}(Z)+Y(C_{22}(Z)-B_{22}(YZ)),\label{eqn20}\\
0&=&B_{22}(Y)Z\label{eqn21}
\end{eqnarray}
for any $Y$ in $B(M^{\bot},M)$ and $Z$ in
$alg\mathscr{N}_{M^{\bot}}$. Putting $Z=I_{M^{\bot}}$ in
Eq.~(\ref{eqn21}), we have
\begin{eqnarray*}
B_{22}(Y)=0\label{eqn22}
\end{eqnarray*}
for any $Y$ in $B(M^{\bot},M)$. Taking $Z=I_{M^{\bot}}$ in
Eq.~(\ref{eqn20}), we get
$A_{12}(I_{M^{\bot}})+C_{12}(I_{M^{\bot}})+YC_{22}(I_{M^{\bot}})=0$
for any $Y\in B(M^{\bot},M)$. So $C_{22}(I_{M^{\bot}})=0$. Since
$C_{22}$ is a derivable mapping at $0$, $C_{22}$ is a derivation on
$alg\mathscr{N}_{M^{\bot}}$(see \cite{Jing}). Thus $C_{22}$ is
inner, and so there is an operator $C\in alg\mathscr{N}_{M^{\bot}}$
such that
\begin{eqnarray*}
C_{22}(Z)=ZC-CZ\label{eqn23}
\end{eqnarray*}
for any $Z$ in $alg\mathscr{N}_{M^{\bot}}$.\\

\emph{Step 5.} For arbitrary idempotent $Q$ in $alg\mathscr{N}_M$
and $Y$ in $B(M^{\bot},M)$, we write $Q_{\lambda}$ for $Q+\lambda
I_{M}$. Obviously there exist two complex numbers
${\lambda}_{1},{\lambda}_{2}$ such that
${\lambda}_{1}+{\lambda}_{2}=-{\lambda}_{1}{\lambda}_{2}=-1$. So
$Q_{{\lambda}_{1}}Q_{{\lambda}_{2}}=Q_{{\lambda}_{2}}Q_{{\lambda}_{1}}=I_{M}$
and $Q_{{\lambda}_{1}}+Q_{{\lambda}_{2}}=2Q-I_{M}$. Taking
$S=\left[\begin{array}{ccc}W Q_{{\lambda}_{1}}&-W
Q_{{\lambda}_{1}}Y\\0&0\\\end{array}\right]$ and
$T=\left[\begin{array}{ccc}
Q_{{\lambda}_{2}}&Y\\~0&I_{M^{\bot}}\\\end{array}\right]$, then
$ST=\Omega$. Thus we have
\begin{eqnarray*}
& &\left[\begin{array}{ccc}A_{11}(W)
&A_{12}(W)\\0&0\\\end{array}\right]=\varphi(\Omega)=\varphi(S)T+S\varphi(T)\\
&=&\left[\begin{array}{ccc}A_{11}(W Q_{{\lambda}_{1}})&A_{12}(W
Q_{{\lambda}_{1}})-B_{12}(W
Q_{{\lambda}_{1}}Y)\\0&0\\\end{array}\right]\left[\begin{array}{ccc}
Q_{{\lambda}_{2}}&Y\\~0&I_{M^{\bot}}\\\end{array}\right]\\&
&+\left[\begin{array}{ccc}W Q_{{\lambda}_{1}}&-W
Q_{{\lambda}_{1}}Y\\0&0\\\end{array}\right]\left[\begin{array}{ccc}A_{11}(Q_{{\lambda}_{2}})
&A_{12}(Q_{{\lambda}_{2}})+B_{12}(Y)+C_{12}(I_{M^{\bot}})\\0&0\\\end{array}\right]\\
&=&\left[\begin{array}{ccc}A_{11}(W
Q_{{\lambda}_{1}})Q_{{\lambda}_{2}}+W
Q_{{\lambda}_{1}}A_{11}(Q_{{\lambda}_{2}})&A_{11}(W
Q_{{\lambda}_{1}})Y+A_{12}(W Q_{{\lambda}_{1}})-B_{12}(W
Q_{{\lambda}_{1}}Y)\\&+W
Q_{{\lambda}_{1}}A_{12}(Q_{{\lambda}_{2}})+W
Q_{{\lambda}_{1}}B_{12}(Y)+W Q_{{\lambda}_{1}}C_{12}(I_{M^{\bot}})\\
~&~\\0&0\end{array}\right].
\end{eqnarray*}
Furthermore,
\begin{eqnarray}
\nonumber A_{12}(W)&=&A_{11}(W Q_{{\lambda}_{1}})Y+A_{12}(W
Q_{{\lambda}_{1}})-B_{12}(W Q_{{\lambda}_{1}}Y)\\&&+W
Q_{{\lambda}_{1}}A_{12}(Q_{{\lambda}_{2}})+W
Q_{{\lambda}_{1}}B_{12}(Y)+W
Q_{{\lambda}_{1}}C_{12}(I_{M^{\bot}}).\label{eqn24}
\end{eqnarray}
Interchanging the position of ${\lambda}_{1}$ and ${\lambda}_{2}$ in
Eq.~(\ref{eqn24}), we have
\begin{eqnarray}
\nonumber A_{12}(W)&=&A_{11}(W Q_{{\lambda}_{2}})Y+A_{12}(W
Q_{{\lambda}_{2}})-B_{12}(W Q_{{\lambda}_{2}}Y)\\&&+W
Q_{{\lambda}_{2}}A_{12}(Q_{{\lambda}_{1}})+W
Q_{{\lambda}_{2}}B_{12}(Y)+W
Q_{{\lambda}_{2}}C_{12}(I_{M^{\bot}}).\label{eqn25}
\end{eqnarray}
Subtracting Eq.~(\ref{eqn25}) from Eq.~(\ref{eqn24}), we have
$$A_{11}(W)Y-B_{12}(W Y)+W B_{12}(Y)=0.$$ Adding Eq.~(\ref{eqn24}) to
Eq.~(\ref{eqn25}), we have $$2[A_{11}(W Q)Y-B_{12}(W QY)+W
QB_{12}(Y)]-[A_{11}(W)Y-B_{12}(W Y)+W B_{12}(Y)]=0.$$ It follows
that
$$A_{11}(W Q)Y-B_{12}(W QY)+W QB_{12}(Y)=0.$$
Since every rank one operator in $alg\mathscr{N}_M$ can be
represented as a linear combination of at most four idempotents in
$alg\mathscr{N}_M$(see \cite{Hadwin}), we get that the above
equation is valid for each rank-one operator in $alg\mathscr{N}_M$.
Furthermore, it is valid for every finite rank operator in
$alg\mathscr{N}_M$(see \cite{Erdos}). Therefore, by the Erd\"{o}s
Density Theorem(see \cite{Erdos}), we have
$$A_{11}(W X)Y-B_{12}(W XY)+W XB_{12}(Y)=0$$ for any $X$ in
$alg\mathscr{N}_M$ and $Y$ in $B(M^{\bot},M)$. If we take $X$ in
$alg\mathscr{N}_M$ and $Y$ in $B(M^{\bot},M)$ in the above equation
such that $XY=0$, from Eq.~(\ref{eqn8}) we can get
$$(W XA-AW X)Y+W XB_{12}(Y)=0.$$That is, $X(AY+B_{12}(Y))=0$. By Lemma \ref{l2} (1), we can pick an operator $G$ from $B(M^{\bot})$ such
that $AY+B_{12}(Y)=YG$, i.e., $B_{12}(Y)=YG-AY$ for any $Y$ in
$B(M^{\bot},M)$. Substituting the expressions of $A_{12}$, $B_{12}$,
$C_{12}$, $B_{22}$ and $C_{22}$ into Eq.~(\ref{eqn20}), we can
obtain that $Y[(G-C)Z-Z(G-C)]=0$ for any $Y$ in $B(M^{\bot},M)$ and
$Z$ in $alg\mathscr{N}_{M^{\bot}}$. Thus $G-C$ in
$(alg\mathscr{N}_{M^{\bot}})'$. Thus there exists a complex number
$\lambda$ such that $G-C=-\lambda I_{M^{\bot}}$(The commutant of
nest algebra is trivial.). Finally, we can obtain that
$B_{12}(Y)=Y(C-\lambda I_{M^{\bot}})-AY$. That is,
\begin{eqnarray*}
B_{12}(Y)=YC-AY-\lambda Y\label{eqn26}
\end{eqnarray*}
for any $Y$ in $B(M^{\bot},M)$.\par In summary, we get that
\begin{eqnarray*}
\varphi(\left[\begin{array}{ccc}X&0\\0&0\end{array}\right])&=&\left[\begin{array}{ccc}A_{11}(X)&A_{12}(X)\\0&A_{22}(X)\end{array}\right]=\left[\begin{array}{ccc}XA-AX&XB\\0&0\end{array}\right]\\
&=&\left[\begin{array}{ccc}X&0\\0&0\end{array}\right]\left[\begin{array}{ccc}A&B\\0&C\end{array}\right]-\left[\begin{array}{ccc}A&B\\0&C\end{array}\right]\left[\begin{array}{ccc}X&0\\0&0\end{array}\right],\\
\varphi(\left[\begin{array}{ccc}0&Y\\0&0\end{array}\right])&=&\left[\begin{array}{ccc}B_{11}(Y)&B_{12}(Y)\\0&B_{22}(Y)\end{array}\right]=\left[\begin{array}{ccc}0&YC-AY-\lambda Y\\0&0\end{array}\right]\\
&=&\left[\begin{array}{ccc}0&Y\\0&0\end{array}\right]\left[\begin{array}{ccc}A&B\\0&C\end{array}\right]-\left[\begin{array}{ccc}A&B\\0&C\end{array}\right]\left[\begin{array}{ccc}0&Y\\0&0\end{array}\right]-\lambda \left[\begin{array}{ccc}0&Y\\0&0\end{array}\right],\\
\varphi(\left[\begin{array}{ccc}0&0\\0&Z\end{array}\right])&=&\left[\begin{array}{ccc}C_{11}(Z)&C_{12}(Z)\\0&C_{22}(Z)\end{array}\right]=\left[\begin{array}{ccc}0&-BZ\\0&ZC-CZ\end{array}\right]\\
&=&\left[\begin{array}{ccc}0&0\\0&Z\end{array}\right]\left[\begin{array}{ccc}A&B\\0&C\end{array}\right]-\left[\begin{array}{ccc}A&B\\0&C\end{array}\right]\left[\begin{array}{ccc}0&0\\0&Z\end{array}\right]
\end{eqnarray*}
for any $X$ in $alg\mathscr{N}_{M}$, $Y$ in $B(M^{\bot},M)$ and $Z$
in $alg\mathscr{N}_{M^{\bot}}$. Hence we obtain that
\begin{eqnarray*}
\varphi(\left[\begin{array}{ccc}X&Y\\0&Z\end{array}\right])&=&\left[\begin{array}{ccc}X&Y\\0&Z\end{array}\right]\left[\begin{array}{ccc}A&B\\0&C\end{array}\right]-\left[\begin{array}{ccc}A&B\\0&C\end{array}\right]\left[\begin{array}{ccc}X&Y\\0&Z\end{array}\right]-\lambda
\left[\begin{array}{ccc}0&Y\\0&0\end{array}\right]\\&=&\left[\begin{array}{ccc}X&Y\\0&Z\end{array}\right]\left[\begin{array}{ccc}A+\frac{1}{2}\lambda
I_{M}&B\\0&C-\frac{1}{2}\lambda
I_{M^{\bot}}\end{array}\right]\\&&-\left[\begin{array}{ccc}A+\frac{1}{2}\lambda
I_{M}&B\\0&C-\frac{1}{2}\lambda
I_{M^{\bot}}\end{array}\right]\left[\begin{array}{ccc}X&Y\\0&Z\end{array}\right].
\end{eqnarray*}
Thus $\varphi$ is an inner derivation.\\~\\
\textsl{Case 2.} $\widehat{M}=M^{\bot}$. Then $\Omega$ may be
represented as the following operator matrices relative to the
orthogonal decomposition $H=M\oplus M^{\bot}$:
$$\Omega=\left[\begin{array}{ccc}0&0\\0&W\end{array}\right],$$ where $W$ is an invertible operator in $alg\mathscr{N}_{M^{\bot}}$.
Since the proof is similar to case 1, the sketch of the proof is
given below. The proof is divided into the following six steps:\\

\emph{Step 1.} For arbitrary $Z_{1}, Z_{2}$ in
$alg\mathscr{N}_{M^{\bot}}$ with $Z_{1}Z_{2}=I_{M^{\bot}}$, taking
$S=\left[\begin{array}{ccc} 0&0\\0&WZ_{1}\\\end{array}\right]$ and
$T=\left[\begin{array}{ccc} 0&0\\0&Z_{2}\\\end{array}\right]$, then
$ST=\Omega$. Since $\varphi$ is derivable at $\Omega$, by imitating
the proof of Case 1, we get that $C_{12}(Z)=B'Z$ for any $Z$ in
$alg\mathscr{N}_{M^{\bot}}$, where $B'=C_{12}(I_{M^{\bot}})$. It
follows from Lemma \ref{l1} that there exists an operator $C'$ in
$B(M^{\bot})$ such that $C_{22}(Z)=ZC'-C'Z$ for any $Z$ in
$alg\mathscr{N}_{M^{\bot}}$.\\

\emph{Step 2.} For arbitrary $Z_{1}, Z_{2}$ in
$alg\mathscr{N}_{M^{\bot}}$ with $Z_{1}Z_{2}=I_{M^{\bot}}$ and
$X_{1},X_{2}$ in $alg\mathscr{N}_M$ with $X_{1}X_{2}=0$, taking
$S=\left[\begin{array}{ccc} X_{1}&0\\0&WZ_{1}\\\end{array}\right]$
and
$T=\left[\begin{array}{ccc}X_{2}&0\\0&Z_{2}\\\end{array}\right]$,
then $ST=\Omega$. By Lemma \ref{l3} and imitating the proof of case
1, we may get that $C_{11}(Z)=0$ for any $Z$ in
$alg\mathscr{N}_{M^{\bot}}$. Since $C_{11}$ vanishes on
$alg\mathscr{N}_{M^{\bot}}$, we obtain that $A_{11}$ is derivable at
$0$. It follows from the expression of $C_{12}$ that
$A_{12}(X)=-XB'$ for any $X$ in $alg\mathscr{N}_M$. We also get that
$A_{22}(X)=0$ for any $X$ in $alg\mathscr{N}_M$.\\
\emph{Step 3.} For arbitrary $Z_{1}, Z_{2}$ in
$alg\mathscr{N}_{M^{\bot}}$ with $Z_{1}Z_{2}=I_{M^{\bot}}$ and $Y$
in $B(M^{\bot},M)$, taking $S=\left[\begin{array}{ccc}
0&0\\0&WZ_{1}\\\end{array}\right]$ and
$T=\left[\begin{array}{ccc}0&Y\\0&Z_{2}\\\end{array}\right]$, then
$ST=\Omega$. Furthermore, we get that $B_{22}(Y)=0$ for any $Y$ in
$B(M^{\bot},M)$.\\
\emph{Step 4.} For an arbitrary $Y$ in $B(M^{\bot},M)$, taking
$S=\left[\begin{array}{ccc}
I_{M}&-YW^{-1}\\0&I_{M^{\bot}}\\\end{array}\right]$ and
$T=\left[\begin{array}{ccc}0&Y\\0&W\\\end{array}\right]$, then
$ST=\Omega$. It follows that $B_{11}(Y)=0$ for any $Y$ in
$B(M^{\bot},M)$.\\
\emph{Step 5.} For arbitrary idempotent $Q'$ in
$alg\mathscr{N}_{M^{\bot}}$ and $Y$ in $B(M^{\bot},M)$, we write
$Q'_{\lambda}$ for $Q'+\lambda I_{M^{\bot}}$. Then there exist two
complex numbers ${\lambda}_{1},{\lambda}_{2}$ such that
${\lambda}_{1}+{\lambda}_{2}=-{\lambda}_{1}{\lambda}_{2}=-1$. So
$Q'_{{\lambda}_{1}}Q'_{{\lambda}_{2}}=Q'_{{\lambda}_{2}}Q'_{{\lambda}_{1}}=I_{M^{\bot}}$
and $Q'_{{\lambda}_{1}}+Q'_{{\lambda}_{2}}=2Q'-I_{M^{\bot}}$. Taking
$S=\left[\begin{array}{ccc}I_{M}&Y\\0&WQ'_{{\lambda}_{1}}\end{array}\right]$
and
$T=\left[\begin{array}{ccc}0&-YQ'_{{\lambda}_{2}}\\0&Q'_{{\lambda}_{2}}\end{array}\right]$,
then $ST=\Omega$. It follows that $A_{11}(I_M)=0$.  Since $A_{11}$
is derivable at $0$, there is an operator $A'$ in $B(M)$ such that
$A_{11}(X)=XA'-A'X$ for any $X$ in $alg\mathscr{N}_{M}$. We get that
$B_{12}(Y)Q'-B_{12}(YQ')+YC_{22}(Q')=0$ for any idempotent $Q'$ in
$alg\mathscr{N}_{M^{\bot}}$. Furthermore,
$B_{12}(Y)Z-B_{12}(YZ)+YC_{22}(Z)=0$ for any $Z$ in
$alg\mathscr{N}_{M^{\bot}}$. If we take $YZ=0$, then
$B_{12}(Y)Z+YC_{22}(Z)=0$. By Lemma \ref{l2}(2) and the expression
of $C_{22}$, we may find an operator $D'$ in $B(M)$ such that
$B_{12}(Y)-YC'=D'Y$ for any $Y$ in $B(M^{\bot},M)$. That is
$B_{12}(Y)=YC'+D'Y$ for any $Y$ in $B(M^{\bot},M)$. \\
\emph{Step 6.} For arbitrary $X$ in $alg\mathscr{N}_{M}$ and $Y$ in
$B(M^{\bot},M)$, take $S=\left[\begin{array}{ccc}
X&-XY\\0&W\\\end{array}\right]$ and
$T=\left[\begin{array}{ccc}0&Y\\0&I_{M^{\bot}}\\\end{array}\right]$
, then $ST=\Omega$. It follows from $B'=C_{12}(I_{M^{\bot}})$ and
the expression of $A_{12}$ that $A_{11}(X)Y-B_{12}(XY)+XB_{12}(Y)=0$
and $A'+D'$ in $(alg\mathscr{N}_M)'$(see \cite{Davidson}). Hence
there exists a complex number $\lambda '$ such that $A'+D'=-\lambda
'I_M$. It follows that $B_{12}(Y)=YC'-A'Y-\lambda 'Y$ for any $Y$ in
$B(M^{\bot},M)$. Thus we have
\begin{eqnarray*}
\varphi(\left[\begin{array}{ccc}X&Y\\0&Z\end{array}\right])&=&\left[\begin{array}{ccc}X&Y\\0&Z\end{array}\right]\left[\begin{array}{ccc}A'&-B'\\0&C'\end{array}\right]-\left[\begin{array}{ccc}A'&-B'\\0&C'\end{array}\right]\left[\begin{array}{ccc}X&Y\\0&Z\end{array}\right]-\lambda'
\left[\begin{array}{ccc}0&Y\\0&0\end{array}\right]\\&=&\left[\begin{array}{ccc}X&Y\\0&Z\end{array}\right]\left[\begin{array}{ccc}A'+\frac{1}{2}\lambda'
I_{M}&-B'\\0&C'-\frac{1}{2}\lambda'
I_{M^{\bot}}\end{array}\right]\\&&-\left[\begin{array}{ccc}A'+\frac{1}{2}\lambda'
I_{M}&-B'\\0&C'-\frac{1}{2}\lambda'
I_{M^{\bot}}\end{array}\right]\left[\begin{array}{ccc}X&Y\\0&Z\end{array}\right].
\end{eqnarray*}\\
Thus $\varphi$ is an inner derivation. This completes the proof.
\end{proof}

\subsection*{Acknowledgement.}
The authors wish to give their thanks to the referees for  helpful
comments and suggestions to improve the original manuscript.

\end{document}